\documentclass[12pt,a4paper]{amsart}
\usepackage{amssymb}

\usepackage{graphicx,amssymb,amsfonts,epsfig,amsthm,a4,amsmath,url}
\usepackage[latin1]{inputenc}

\usepackage{amsfonts,amssymb,verbatim}
\usepackage{latexsym}
\usepackage{mathrsfs}
\usepackage{stmaryrd}
\usepackage{xspace}
\usepackage{enumerate, paralist}
\usepackage[usenames,dvipsnames]{color}
 \usepackage{txfonts, pxfonts}

\newtheorem*{thmm}{Theorem}

\newtheorem{thm}{Theorem}[section]
\newtheorem{cor}[thm]{Corollary}
\newtheorem{lem}[thm]{Lemma}
\newtheorem{clai}[thm]{Claim}
\newtheorem{prop}[thm]{Proposition}
\theoremstyle{definition}
\newtheorem{defn}[thm]{Definition}

\theoremstyle{remark}

\numberwithin{equation}{section}

\newcommand{\Z}{\mathbb{Z}}

\newcommand{\E}{\mathbb{E}}

\newcommand{\N}{\mathbb{N}}
\newcommand{\R}{\mathbb{R}}

\newcommand{\Q}{\mathbf{Q}}

\newcommand{\eps}{\varepsilon}

\begin{document}
\title[Speed of convergence]{Speed of convergence in first passage percolation and geodesicity of the average distance}

\author{Romain Tessera}
\address{Laboratoire de Math\'ematiques, B\^atiment 425\\ Universit\'e Paris-Sud 11, 91405 Orsay, France
}
\email{romain.tessera@math.u-psud.fr}
\date{\today}
\subjclass[2010]{46B85, 20F69, 22D10, 20E22 }
\keywords{First passage percolation, limit shape theorem, speed of convergence}

\baselineskip=16pt

\begin{abstract}
We give an elementary proof that Talagrand's sub-Gaussian concentration inequality implies a limit shape theorem for first passage percolation on any Cayley graph of $\Z^d$, with a speed of convergence   $\lesssim\left(\frac{\log n}{n}\right)^{1/2}$. This slightly improves Alexander's bounds from \cite{Al}. Our approach, which does not use the subadditive theorem, is based on proving that the average distance $\E d_{\omega}$ on $\Z^d$ is close to being geodesic. Our key observation, of independent interest, is that the problem of estimating the rate of convergence for the average distance is equivalent (in a precise sense) to estimating its ``level of geodesicity".
\end{abstract}

\maketitle
\tableofcontents


\section{Introduction}
First passage percolation (FPP) is a way to randomly perturb the distance on a connected graph. Let us recall how this random process is defined.

Consider a connected non-oriented graph $X$, whose set of vertices (resp.\ edges) is denoted by $V$ (resp.\ $E$). We first define the notion of {\it weighted} graph metric on $V$.
For every function $\omega:E\to (0,\infty)$, we equip $V$ with the weighted graph metric $d_{\omega}$, where each edge $e$ has weight $\omega(e)$. In other words, for every $x,y\in V$, $d_{\omega}(x,y)$ is defined as the infimum over all path $p=(e_1,\ldots, e_m)$ joining $x$ to $y$ of $\ell_f(p):=\sum_{i=1}^m\omega(e_i)$. Denote by $d$ the graph metric on $V$,  corresponding to the constant function  $\omega=1$.

Let $\nu$ be a probability measure supported on $[0,\infty)$. The random metric of {\em first passage percolation} consists in choosing the weight $\omega(e)$ independently  according to $\nu$.  Note that $\E d_{\omega}(x,y)$ defines a distance on $V$, that we call the {\it average distance} and denote by $\bar{d}(x,y)$.

A central result in FPP is the following Gaussian concentration inequality due to Talagrand.

\begin{thmm}\cite[Proposition 8.3]{Talagrand}).
Suppose that $\omega(e)$ has an exponential moment: i.e.\ there exists $c>0$ such that $\E \exp(c\omega(e))<\infty$. Then there exists $C_1$ and $C_2$ such that for every graph $X=(V,E)$, for every pair of vertices $x,y$, and for every $u\geq 0$,
\begin{equation}\label{eq:expmoment}
P\left(|d_{\omega}(x,y)-\bar{d}(x,y)|\geq u\right)\leq C_1\exp\left(-C_2\min\left\{\frac{u^2}{d(x,y)},u\right\}\right).
\end{equation}
\end{thmm}

\subsection{A quantitative limit shape theorem for Cayley graphs of $\Z^d$}


\

\noindent{\bf Basic assumptions.}
In order to avoid useless repetitions, let us once and for all list the technical assumptions on the edge's length distribution $\nu$, that will be required in most of our statements.

\begin{itemize}
\item {\bf (A1)}
We assume that $\nu$ has an exponential moment, and therefore satisfies (\ref{eq:expmoment}) for some constants $C_1$ and $C_2$ (this assumption can probably relaxed but we choose not to focus on this aspect here). 
\item {\bf  (A2)}
We  also suppose that there exists $a>0$ such that $\bar{d}(x,y)\geq ad(x,y)$ for all $x,y\in V$. 
\end{itemize}

When one works with the standard Cayley graph of $\Z^d$, the second assumption is satisfied exactly when $\nu(\{0\})<p_c$, where $p_c$ is the critical probability of percolation on $\Z^d$ \cite{Ke}.
For more general graphs, we show that the second condition is fulfilled provided that 
 $\nu(\{0\})<1/k$, where $k$ is an upper bound on the degree of the graph (see Corollary \ref{cor:degree}).  
 
Observe that by triangular inequality, $\bar{d}\leq (\E \omega(e))d$.  In the sequel we denote $b:=\E \omega(e).$
It follows that under our second assumption, $d$ and $\bar{d}$ are actually bi-Lipschitz equivalent, more precisely,
\begin{equation}\label{eq:biLip}
ad\leq \bar{d}\leq bd.
\end{equation}

We shall adopt the following notation:  given $v\in V$ and $r>0$, let $\bar{B}(v,r)$ (resp.\ $B_{\omega}(v,r)$) denote the ball of radius $r$ for the average distance $\bar{d}$ (resp.\ for the random distance $d_{\omega}$).

Our main result is the following theorem.


\begin{thm}\label {thm:Main}{\bf (quantitative asymptotical shape theorem)}
We consider a Cayley graph of $\Z^d$, associated to some finite generating subset. We assume $(A_1)$ and $(A_2)$ are satisfied.
There exists a norm $\|\cdot\|$ on $\R^d$ such that
for a.e.\ $\omega$, there exists $C>0$ and $n_0$ such that for all $n\geq n_0$, 

\begin{equation}\label{eq:balls}
B_{\|\cdot\|}\left(0, n -C\left(n\log n\right)^{1/2}\right)\cap \Z^d\subset B_{\omega}(0,n)\subset B_{\|\cdot\|}\left(0, n +C\left(n\log n\right)^{1/2}\right).
\end{equation}
\end{thm}
The fact that the rescaled ball converges to a convex body was first proved by Kesten \cite{Ke}, extending previous work by Richardson \cite{R} and Cox-Durrett \cite{CD} (for background see \cite{GK} \cite{Ke}  \cite{Ah}). The first quantitative estimates, given by Kesten \cite{Ke'}, depended on the dimension. These estimates were later improved by Alexander \cite{Al} who proved an error term in $O(n^{1/2}\log n)$. More recently, a Gaussian estimate for the lower tail has recently been obtained under a quadratic moment condition in \cite{DK} (see also \cite{Ku,Z,Z'}). Following the strategy of proof of \cite{Al} (itself inspired from \cite{Ke'}) they manage to deduce the right-hand side inequality of (\ref{eq:balls}) under this low moment condition. 

Although Theorem \ref{thm:Main} seems to be new, it only represents a modest improvement of the main result of  \cite{Al}, and is likely to remain far from optimal. Indeed, we recall that physicists believe that the error term for $\Z^2$ should be $n^{1/3}$. However, it is not clear what rate should be expected in higher dimensions (see \cite{Ke'} for a more detailed account and the relevant references).

Maybe more interesting than the result itself is its (self-contained) proof. On the one hand, 
it breaks down the main result into two independent statements that we shall describe below: one is a straightforward bound on the fluctuations about the average distance, while the other one bounds the speed of convergence of the rescaled average distance to the limit norm. This last step can also be decomposed into two independent results: an easy one which is based on $(A_1)$ and is valid for any graph with polynomial growth, and  a more subtle, purely geometric statement that explicitly uses the abelian group structure of $\Z^d$.

Another interesting feature of this new approach is the fact that it does not use the subadditive ergodic theorem, unlike the previous ones. In \cite{BT}, we exploit this to obtain a limit shape theorem for any Cayley graphs of polynomial growth, which did not seem to be approachable by previous methods.

\subsection{Fluctuations around the average metric / Speed of convergence for the average metric}

We now describe the two main estimates that are needed in our proof of Theorem \ref{thm:Main}. Let us start with the following straightforward consequence of Talagrand's concentration inequality. 
\begin{prop}\label{prop:fluctuations}{\bf (Fluctuations about the average distance)}
Let $d>0$ and $K>0$, and let  $r_n\in \N$ be an increasing sequence. We assume $(A_1)$ and $(A_2)$ are satisfied.  Then there exists $D>0$ such that the following holds.
Let $X=(V,E)$ be a graph and let $o_n$ be a sequence of vertices such that  $|B(o_n,r_n)|\leq Kr_n^d.$ Then for a.e.\ $\omega$, there exists $n_0$ such that for $n\geq n_0$,
\begin{equation}\label{eq:uniformVariance}
sup_{x,y\in B(o_n,r_n)}|d_{\omega}(x,y)-\bar{d}(x,y)|\leq C(r_n\log r_n)^{1/2}.
\end{equation}
\end{prop}
We deduce from the previous proposition that there exists $C'$  such that for a.e.\ $\omega$ there exists $n_0$ such that for all $n\geq n_0$, one has 
$$\bar{B}\left(0,r_n-C'(r_n\log r_n)^{1/2}\right)\subset B_{\omega}(0,r_n)\subset \bar{B}\left(0,r_n+C'(r_n\log r_n)^{1/2}\right).$$

The complementary (and main) step in the proof of  Theorem \ref{thm:Main} therefore consists in estimating the speed of convergence of the rescaled ball for the average distance on $\Z^d$.  
\begin{thm}\label {prop:abeliancase}{\bf (Asymptotical shape theorem for the average distance)}
We consider a Cayley graph of $\Z^d$. We assume $(A_1)$ and $(A_2)$ are satisfied.
There exists $C>0$ and $n_0$ such that for all $n\geq n_0$, 
$$B_{\|\cdot\|}\left(0,n-C(n \log n)^{1/2}\right)\cap \Z^d\subset \bar{B}(0,n)\subset B_{\|\cdot\|}\left(0,n+C(n \log n)^{1/2}\right).$$
\end{thm}

\subsection{On how to quantify being geodesic ``in an asymptotical way"}
In order to explain the strategy behind the proof of Theorem \ref{prop:abeliancase}, we need to introduce the notion of {\it strong asymptotical geodesicity}. 
Before giving a formal definition, let us review two important properties of a (discrete) geodesic metric space. Recall that a graph satisfies the following two (equivalent) properties:
\begin{itemize}
\item for any $x,y\in X$, such that $d(x,y)=n$, there exists a sequence $x=x_0,x_1,\ldots,x_n=y$ such that $d(x_{i},x_{i+1})=1$ for all $0\leq i<n$;
\item for every $n<n'$ and every $x\in X$, the distance from any point in $B(x,n')$ to $B(x,n)$ is at most $n'-n$. This can also be formulated as 
$$B(x,n')\subset [B(x,n)]_{n'-n},$$
where $[A]_t$ denotes the $t$-neighborhood of the subset $A$.
\end{itemize}
This suggests at least  two ways of defining being geodesic ``in an asymptotical way'':

The first one is called ``inner metric" in \cite{Pa}, or ``asymptotical geodesic metric" \cite{Br}. The space $X$ is asymptotically geodesic if for all $\eps>0$ there exists $\alpha$ such that for all $x,y\in X$, there exists a sequence $x=x_0,\ldots x_m=y$ such that  $d(x_i,x_{i+1})\leq \alpha$, and $\sum_{i=0}^{m-1}d(x_i,x_{i+1})\leq (1+\eps)d(x,y).$

The second one is ``monotone geodesic metric" as defined by the author in \cite{Te}. 
Monotone geodesicity is defined by requiring the existence of a constant $T$ such that for all $x$ and all $r$,
$B(x,r+1)\subset [B(0,r)]_T.$ One can make the latter ``asymptotical" by requiring $T$ to be an unbounded function of $r$.

Asymptotical geodesicity was used by Pansu to obtain a limit shape theorem for Cayley graphs of nilpotent groups \cite{Pa}, while monotone geodesicity was used to bound the size of the spheres in graphs with the doubling property. In some sense both notions have to do with controlling the error terms when estimating the size of large balls. It is therefore not surprising that the notion that we need  here is a  quantitative combination of these two. 

\begin{defn}\label{defn:strongasympt} {\bf (Strongly Asymptotically Geodesic spaces)} Let  $N:\R_+\to \R_+$ be an increasing function such that $\lim_{\alpha\to \infty} N(\alpha)=\infty$. A metric space $X$ is called SAG(N)   if there exists $\alpha_0\geq 0$ such that for all  integer $m\geq 1$, and for all $x,y\in X$ such that $d(x,y)/m\geq \alpha_0$,  there exists  a sequence $x=x_0\ldots, x_m=y$ satisfying, for all $0\leq i\leq m-1,$
\begin{equation}\label{eq:strong}
\frac{d(x,y)}{m}\left(1-\frac{1}{N(\alpha)}\right)\leq d(x_i,x_{i+1})\leq \frac{d(x,y)}{m}\left(1+\frac{1}{N(\alpha)}\right),
\end{equation} 
where $\alpha=d(x,y)/m$;
and  for all large enough $r$,
\begin{equation}\label{eq:strong'}
B\left(0,\left(1+\frac{1}{N(r)}\right)r\right)\subset [B(0,r)]_{\frac{6r}{N(r)}}.
\end{equation}
\end{defn}

As nicely suggested to me by Xuan Wang, this quite complicated condition can be deduced (for certain types of functions $N$) from the following much simpler property, which is another very natural asymptotic version of geodesicity.

\begin{defn}\label{defn:asymptoticConvexity}
Given a increasing function $N:\R_+\to \R_+$ such that $\lim_{\alpha\to \infty} N(\alpha)=\infty$, we say a metric space $X$ is SAG*(N) if there exists $\alpha_0>0$ such that the following holds: for all $x,y\in X$ satisfying $d(x,y)=r\geq \alpha_0$, and $\lambda\in [0,1]$, we can find $z\in X$ such that 
$$\left(\lambda-\frac{1}{N(r)}\right)d(x,y)\leq d(x,z)\leq \left(\lambda+\frac{1}{N(r)}\right)d(x,y), $$
and
$$\left(1-\lambda-\frac{1}{N(r)}\right)d(x,y)\leq d(x,z)\leq \left(1-\lambda+\frac{1}{N(r)}\right)d(x,y).$$
\end{defn}

The connection between these two notions is given by the following proposition.

\begin{prop}\label{prop:NewStrongImpliesOldStrong}{\bf (SAG* implies SAG)}
Let $u,c>0$, $v\in \R$, and consider the function $N(\alpha)=c\alpha^u (\log \alpha)^v$ (defined for $\alpha>1$).  
If a metric space $X$ is SAG*(N), then it is SAG(c'N) for some $c'>0$.
\end{prop}

\subsection{Strong asymptotic geodesicity of the average metric}
The proof of Theorem \ref{prop:abeliancase} relies on the following result, of independent interest.
Observe that the only geometric property of the graph that is required is some sub-polynomial volume growth condition\footnote{Since we suppose $(A_2)$, it is implicitly assumed that the degree of $X$ is bounded.}. It  applies for instance to first passage percolation on fractal graphs, Cayley graphs of nilpotent groups (see \cite{BT}), or random environments such as the infinite cluster of subcritical percolation on $\Z^d$.

 \begin{prop}\label{prop:asymptoticallyGeodesic}{\bf (The average metric is SAG*($(\alpha/\log \alpha)^{1/2}$))}
$X=(V,E)$ be a graph. Suppose that there exists $K>0$ and $d>0$ such that for all $o\in V$ and all $r>0$, 
$|B(o,r)|\leq Kr^d$. We assume $(A_1)$ and $(A_2)$ are satisfied. Then there exists a constant $c>0$ such that the metric $\bar{d}$ is SAG*(N) with $$N(\alpha)= c\left(\frac{\alpha}{\log \alpha}\right)^{1/2}.$$
In particular it is SAG(c'N) for some $c'>0$.  
\end{prop}

The idea behind the proof of Proposition \ref{prop:asymptoticallyGeodesic} is relatively simple: it consists in exploiting the concentration inequality due to Talagrand to show that $\bar{d}$ being close to $d_{\omega}$ with large probability, since $d_{\omega}$ is geodesic, then $\bar{d}$ cannot be too far from being geodesic. 

Let us be more precise about what we mean by ``$d_{\omega}$ is geodesic": by definition, for all $x,y\in V$ there exists a simple path $p=(e_1,\ldots,e_k)$ in $X$ joining $x$ to $y$, and such that $d_{\omega}(x,y)=\sum_{i}\omega(e_i)$. Such a path will be called an $\omega$-geodesic between $x$ and $y$. 

\subsection{A geometric statement about invariant metrics on $\Z^d$}
So far we have not used any specific feature of $\Z^d$. It turns out that the connection between strong asymptotical geodesicity of the average metric and the limit shape theorem follows from a very general result about invariant metrics on $\Z^d$.

In the sequel, an increasing function $\phi:\R_+\to \R_{+}$ is called sublinearly doubling if there exists a function $\eta: \R_+\to \R_{+}$ satisfying   $\lim_{\lambda\to \infty}\eta(\lambda)/\lambda=0$ such that for all $\lambda>0$, $\phi(\lambda r)\leq \eta(\lambda)\phi(r)$. 

\begin{prop}\label{prop:AGversusCV}{\bf (Strong Asymptotical Geodesicity versus Limit Shape)}
We let $\delta$ be some invariant metric on $\Z^d$. We let $\phi:\R_+\to [1,+\infty)$ be an increasing, sublinearly doubling function. The following two assertions are equivalent. 
\begin{itemize}
\item[(i)] There exists a constant $c>0$ such that $\delta$ is SAG(N) with 
$$N(\alpha)\geq c\phi(\alpha),$$
for $\alpha$ large enough. 
\item[(ii)] There exists a norm $\|\cdot \|$ on $\R^d$ and $C>0$ such that for all large enough $n$, 
$$B_{\|\cdot\|}\left(0,n-\frac{Cn}{\phi(n)}\right)\cap \Z^d\subset B_{\delta}(0,n)\subset B_{\|\cdot\|}\left(0,n+\frac{Cn}{\phi(n)}\right).$$
\end{itemize}
\end{prop}

It is now clear that Theorem \ref{prop:abeliancase} results from Proposition \ref{prop:AGversusCV} (for $\delta=\bar{d}$, and $\phi(\alpha)=(\alpha/\log \alpha)^{1/2}$) and Proposition \ref{prop:asymptoticallyGeodesic}. 

We hope that Proposition \ref {prop:AGversusCV} will be useful for future attempts to improve the known  estimates --both from above and from below-- on the speed of convergence in Theorem \ref{prop:abeliancase}.

\subsection*{Organization} In \S \ref{sec:uniformVariance}  we provide the quick proof of Proposition \ref{prop:fluctuations}. Section \ref{sec:SAG} is dedicated to the proof of Proposition \ref{prop:asymptoticallyGeodesic}. These first two short sections are the only ones concerned with probabilistic arguments (recall that these statements are proved for general graphs with a sub-polynomial growth condition). 
Section \ref{sec:SAG*SAG} is dedicated to the proof of Proposition \ref{prop:NewStrongImpliesOldStrong}, which holds for any metric space. In the last two sections (\S \ref{sec:MainProof} and \S \ref{sec:(ii)to(i)}) we prove both implications of Proposition \ref{prop:AGversusCV}, which is a statement about invariant metrics on $\Z^d$. The different sections can be read independently.

\subsection*{Aknowledgement} I am grateful to Itai Benjamini for attracting my attention to this beautiful subject and for many inspiring discussions. 
I would like to thank Xuan Wang for suggesting a nice alternative notion of asymptotic geodesicity, and for his numerous remarks and corrections.  

\section{Proof of Proposition \ref{prop:fluctuations}}\label {sec:uniformVariance}
Let $D$ be a constant to be determined later. Applying Talagrand's theorem, we obtain that for all large enough $r$, and all $x,y$ such that $d(x,y)\leq r$,
$$P\left(|d_{\omega}(x,y)-\bar{d}(x,y)|^2\geq Dr\log r\right)\leq C_1 \exp\left(-C_2 D\log r\right).$$
Now, letting $D=(2d+2)/(C_2b^2)$, we deduce that for all large enough $r$, and all $x,y$ such that $d(x,y)\leq r$,
$$P\left(|d_{\omega}(x,y)-\bar{d}(x,y)|^2\geq Dr\log r\right)\leq C_1r^{-2d-2}.$$
Hence for $n$ large enough,
$$P\left(\sup_{x,y\in B(o_n,r_n)}|d_{\omega}(x,y)-\bar{d}(x,y)|^2\geq Dr_n\log r_n\right)\leq C_1r_n^{-2d-2}|B(o_n,r_n)|^2\leq C_1Kr_n^{-2}.$$
Proposition \ref {prop:fluctuations} now follows from the fact that $\sum_nr_n^{-2}<\infty$ (with C=D).
\qed

\section{The average metric is SAG}\label{sec:SAG}

The goal of this section is to prove Proposition \ref{prop:asymptoticallyGeodesic}, which immediately follows from the following proposition. 

\begin{prop}\label{pop:k-geodesic}
Let $d>0$ and $K>0$, and let  $r_n\in \N$ be an increasing sequence. We assume $(A_1)$ and $(A_2)$ are satisfied.  Then there exists $c>0$ and $n_0$ such that the following holds.
Let $X=(V,E)$ be a graph and let $o_n$ be a sequence of vertices such that  $|B(o_n,r_n)|\leq Kr_n^d.$  Then for all $x,y\in B(o_n,ar_n/(32b))$ and for all $0\leq \lambda\leq 1$, there exists a vertex $z\in B(o_n,r_n)$ such that for $n\geq n_0$,
$$\left|\lambda\bar{d}(x,y)-\bar{d}(x,z)\right|\leq \frac{r_n}{N(r_n)},$$
and 
$$\left|(1-\lambda)\bar{d}(x,y)-\bar{d}(z,y)\right|\leq \frac{r_n}{N(r_n)}$$
where $$N(r)= c\left(\frac{r}{\log r}\right)^{1/2}.$$
\end{prop}

\begin{proof}
In the the proof of  Proposition \ref{prop:fluctuations}, we established that 
$$P\left(\sup_{x,y\in B(o_n,r_n)}|d_{\omega}(x,y)-\bar{d}(x,y)|^2\geq Dr_n\log r_n\right)\leq C_1K r_n^{-2},$$
with $D=(2d+2)/(C_2b^2)$ (remember that $C_1$ and $C_2$ are the constants appearing in the conclusion of Talagrand's theorem).
Let $n_0$ be the smallest integer so that $Er_n^{-2}<1$. Then for all $n\geq n_0$, there exists $\omega$ (depending on $n$) such that 
\begin{equation}\label{eq:fluctuation}
\sup_{z_1,z_2\in B(o,r_n)}\left|d_{\omega}(z_1,z_2)-\bar{d}(z_1,z_2)\right|\leq D (r_n\log r_n)^{1/2}.
\end{equation}
Assume in addition that $n_0$ is large enough so that $D (r_n\log r_n)^{1/2}\leq ar_n/16$ for all $n\geq n_0$.
Let $\gamma$ be some $\omega$-geodesic between $x$ and $y$. First of all, note  that $\gamma$ cannot escape from the ball $B(o_n,r_n)$. Indeed, suppose there is $1\leq i\leq k$ such that $d(o_n,\gamma(i))= r_n$, then by triangular inequality, $d(x,\gamma(i))\geq r_n/2$, hence $\bar{d}(x,\gamma(i)\geq ar/2$. So (\ref{eq:fluctuation})  implies that 
$$d_{\omega}(x,\gamma(i)) \geq \bar{d}(x,\gamma(i))-D (r_n\log r_n)^{1/2}\geq ar_n/2-ar_n/16\geq ar_n/4.$$ 
which contradicts the fact that  
$$d_{\omega}(x,\gamma(i))\leq d_{\omega}(x,y)\leq \bar{d}(x,y)+ar_n/16\leq ar_n/8.$$ 

By (\ref{eq:fluctuation}), the maximum of $\omega(e)$ over all edges on $\gamma$ is at most $b+D (r_n\log r_n)^{1/2}\leq D'(r_n\log r_n)^{1/2}$ for some $D'>0$. Therefore, one can find a vertex $z$ in $\gamma$ such that 
$$\left|\lambda d_{\omega}(x,y)-d_{\omega}(x,z)\right|\leq D'(r_n\log r_n)^{1/2},$$
and 
$$\left|(1-\lambda)d_{\omega}(x,y)-d_{\omega}(z,y)\right|\leq D' (r_n\log r_n)^{1/2}.$$

But then combining these inequalities with (\ref{eq:fluctuation}), we get 
$$\left|\lambda\bar{d}(x,y)-\bar{d}(x,z)\right|\leq 4D' (r_n\log r_n)^{1/2},$$
and 
$$\left|(1-\lambda)\bar{d}(x,y)-\bar{d}(z,y)\right|\leq 4D'(r_n\log r_n)^{1/2},$$
so that the proposition follows with $c=1/(4D')$.
\end{proof}
\section{Proof of Proposition \ref{prop:NewStrongImpliesOldStrong}}\label{sec:SAG*SAG}

The proof of (\ref{eq:strong'})  only relies on the assumption that  $N$ is increasing and unbounded. We assume that $r\geq 1$ is large enough so that $N(r)\geq 1$. 
Let $y\in B\left(0,\left(1+\frac{1}{N(r)}\right)r\right)$.
Applying SAG*(N) with $\lambda=1-2/N(r)$ yields some let $z\in X$ such that
$$d(x,z)\leq \left(1-\frac{1}{N(r)}\right)d(x,y)\leq  \left(1-\frac{1}{N(r)}\right)\left(1+\frac{1}{N(r)}\right)r\leq r;$$
and 
$$d(x,z)\leq \frac{3d(x,y)}{N(r)}\leq \frac{6r}{N(r)}.$$
 Hence (\ref{eq:strong'}) follows.

Let us turn to the proof of (\ref{eq:strong}). There, we assume that $N(\alpha)=c\alpha^u\log \alpha^v$ for $c>0, u>0$ and $v\in \R$. 
Actually we shall prove a stronger statement:
\begin{prop}\label{prop:VSG}
Assuming that $X$ is SAG*(N),
there exists $\alpha_0\geq 0$ such that for all sequence $\lambda_0=0< \lambda_1<\ldots<\lambda_m=1$ and for all $x,y\in X$ such that $\Delta:=\min_{0\leq i<m}(\lambda_{i+1}-\lambda_i)d(x,y)\geq \alpha_0$,  there exists a sequence $x=x_0\ldots, x_m=y$ satisfying, for all $0\leq i\leq m-1,$
$$\left(1-\frac{C_1}{N(\Delta)}\right)(\lambda_{i+1}-\lambda_i) d(x,y)\leq d(x_i,x_{i+1})\leq \left(1+\frac{C_1}{N(\Delta)}\right)(\lambda_{i+1}-\lambda_i) d(x,y).$$
\end{prop}
For the sake of concreteness and since this is the only case we really need in the sequel, we shall assume that $u=-v=1/2$ (the general case is proved in exactly the same way).

We first state the following immediate consequence of the definition of strong $N$-asymptotic geodesicity.

\begin{lem}\label{lem:additiveToMultiplicative} 
Assuming that $X$ is SAG*(N), there exists $\alpha_0 \geq 0$ such that for all $\lambda_0\in (0,1/2]$, all $\lambda\in [\lambda_0,1-\lambda_0]$,
and all $x,y\in X$, there exists $z\in X$ such that 
$$\left(1-\frac{1}{\lambda_0 N(r)}\right)\lambda r\leq d(x,z)\leq \left(1+\frac{1}{\lambda_0 N(r)}\right)\lambda r,$$
and
$$\left(1-\frac{1}{(1-\lambda_0)N(r)}\right)\left(1-\lambda\right)r\leq d(x,z)\leq \left(1+\frac{1}{(1-\lambda_0)N(r)}\right)\left(1-\lambda\right)r,$$
where $r=d(x,y)$.
\end{lem}

We start proving a special case of Proposition \ref{prop:NewStrongImpliesOldStrong} where $m=2^k$. 

\begin{lem}\label{lem:biadicSG}
Assuming that $X$ is  SAG*(N),
there exists $\alpha_0 \geq 0$, and $C_1'$ such that for all  integer $k\geq 1$, and for all $x,y\in X$ such that $d(x,y)/2^k\geq \alpha_0$,  there exists  a sequence $x=x_0\ldots, x_{2^k}=y$ satisfying, for all $0\leq i\leq 2^k-1,$
$$\frac{d(x,y)}{2^k}\left(1-\frac{C_1'}{N(d(x,y)/2^{k})}\right)\leq d(x_i,x_{i+1})\leq \frac{d(x,y)}{2^k}\left(1+\frac{C_1'}{N(d(x,y)/2^{k})}\right).$$
\end{lem}
\begin{proof}We let $r_0=r'_0:=d(x,y)$.
We let $n\in \N$ be such that $2^n< r_0\leq 2^{n+1}$.   Assuming that $n$ is large enough  so that $2^{n}\geq \alpha_0$, where $\alpha_0$ is the parameter appearing in the definition SAG*, there exists $z$ such that 
$$r_0/2-C(r_0\log r_0)^{1/2}\leq d(x,z),d(z,y)\leq r_0/2+C(r_0\log r_0)^{1/2},$$ for some constant $C$. 
We let $r_1=\max\{d(x,z),d(z,y)\}$ and $r'_1=\min\{d(x,z),d(z,y)\}$ and apply SAG* to $(x,z)$ and $(z,y)$. Continuing this subdivision process as long as $r'_{k-1}\geq \beta$, we find a sequence $r_1,\ldots, r_k,\ldots,$ satisfying 
\begin{equation}\label{eq:rk}
r_{k}\leq \frac{1}{2}\left(r_{k-1}+C(r_{k-1}\log r_{k-1})^{1/2}\right),
\end{equation}
and 
\begin{equation}\label{eq:rk}
r_{k}'\geq \frac{1}{2}\left(r'_{k-1}-C(r'_{k-1}\log r'_{k-1})^{1/2}\right),
\end{equation}
and a sequence of finite sequences of vertices $x=z_0(k),\ldots,z_{2^{k}}(k)=y$ such that $$r'_k\leq d(z_i(k),z_{i+1}(k))\leq r_{k},$$
for all $0\leq i<2^{k}-1.$

\begin{clai}\label{lem:biadic}
There exists a constant $A$ such that for all $k$ such that $r'_k\geq \beta$, $$A^{-1}2^{-k}r_0\leq r'_k\leq r_k\leq A2^{-k}r_0.$$
\end{clai}
\begin{proof}
Let us first prove the right inequality, the other one being similar. Let $k\geq 2$, and observe that $$r_k\leq \frac{1}{2}\left(r_{k-1}+Cr_{k-1}^{2/3}\right).$$  We do the following change of variable: $A_k=2^{-k}r_k$ (note that $A_k\geq 1$).  We have $$A_k\leq A_{k-1}+CA_{k-1}^{2/3}2^{-k/3}\leq A_{k-1}\left(1+C2^{-k/3}\right),$$ 
from which we easily deduce that $A_k$ is bounded by some $A$ only depending on $C$.
\end{proof}
In what follows, we assume that $k$ is such that $r'_k\geq \beta$.
We deduce from the lemma and from the fact that $r_k\geq r_0/2^k\geq 2^{n-k}$ (which follows by triangular inequality) that 
\begin{eqnarray*}
r_{k} & \leq & \frac{1}{2}\left(r_{k-1}+C(r_{k-1}\log r_{k-1})^{1/2}\right)\\
         & \leq & \frac{r_{k-1}}{2}\left(1+C\left(\frac{\log r_{k-1}}{r_{k-1}}\right)^{1/2}\right)\\
         & \leq & \frac{r_{k-1}}{2}\left(1+C(n-k+A+1)^{1/2}2^{-(n-k)/2}\right)\\
         & \leq & \frac{r_{k-1}}{2}\left(1+C'(n-k)^{1/2}2^{-(n-k)/2}\right)
\end{eqnarray*}
for some constant $C'$.  
Taking the log and using that $\log(1+x)\leq x$ for $x\geq 0$, we have 
\begin{eqnarray*}
\log (2^kr_{k}/r_0) & \leq & C'\sum_{i=1}^k(n-i)^{1/2}2^{-(n-i)/2}\\
                             & \leq & C'\sum_{j\geq n-k}j^{1/2}2^{-j/2}\\
                             & \leq & C"(n-k)^{1/2}2^{-(n-k)/2},
\end{eqnarray*}
for some constant $C">0$. Remember that $r_k'$, and therefore $A2^{n-k}$ is supposed to be larger than $\alpha_0$. Up to enlarging $\alpha_0$ if necessary we can assume that  $C"(n-k)^{1/2}2^{-(n-k)/2}\leq 1$. Then, using that $\exp(x)\leq 1+2x$, for all $0\leq x\leq 1$, we deduce that there exists a constant $C$ such that
\begin{equation}\label{eq:rkabove}
r_k\leq 2^{-k}\left(1+2C"(n-k)^{1/2}2^{-(n-k)/2}\right)r_0\leq \alpha\left(1+C\left(\frac{\log \alpha}{\alpha} \right)^{1/2}\right),
\end{equation}
where $\alpha=r_0/2^k$.
We prove similarly that 
\begin{equation}\label{eq:rk'below}
r'_k\geq \alpha\left(1-C\left(\frac{\log \alpha}{\alpha} \right)^{1/2}\right).
\end{equation}
We let $x=x_0=z_0(k),\ldots, x_{2^k}=z_{2^k}(k)=y$. We deduce from (\ref{eq:rkabove}) and (\ref{eq:rk'below}) that there exists a constant $C$ such that for every $0\leq i\leq 2^k-1$,
$$\left(1-C\left(\frac{\log \alpha}{\alpha} \right)^{1/2}\right)d(x,y)/2^k\leq d(x_i,x_{i+1})\leq \left(1-C\left(\frac{\log \alpha}{\alpha} \right)^{1/2}\right)d(x,y)/2^k,$$
where $\alpha=r_0/2^k$. So Lemma \ref{lem:biadicSG} follows. 
\end{proof}

\medskip

\noindent{\bf Proof of Proposition \ref{prop:VSG}.}
It is now easy to deduce Proposition \ref{prop:VSG} from Lemma \ref{lem:biadicSG}. Indeed, choose $k$ such that $2^{-k+2}d(x,y)\leq \Delta\leq 2^{-k+3}d(x,y)$ and assume that a sequence $x_0=x,x_1,\ldots, x_{2^k}=y$ as in Lemma \ref{lem:biadicSG} has been constructed. 
Let $0=\nu_0<\mu_1<\nu_1<\mu_2<\nu_2<\ldots<\nu_{m-1}<\mu_m=1$ be an increasing sequence of elements of $2^{-k}\N$ such that for every $1\leq i\leq m-1$, $\lambda_i$ is a convex combination $t_i\nu_i+(1-t_i)\mu_i$, with $1/3\leq t_i\leq 2/3$ (it is easy to see that this can be done). 
For every $1\leq i\leq m-1$, we apply Lemma \ref{lem:additiveToMultiplicative} with $x=x_{\mu_i}, y=x_{\nu_{i+1}}$ and $\lambda=t_i$,  yielding a constant $C'$ and $z_i$ such that 
$$\left(1-C'\left(\frac{\log \Delta}{\Delta}\right)^{1/2}\right)t_i d(x_{\mu_i},x_{\nu_i})\leq d(x_{\mu_i},z_i)\leq \left(1+C'\left(\frac{\log \Delta}{\Delta}\right)^{1/2}\right)t_i d(x_{\mu_i},x_{\nu_i}),$$ 
and 
$$\left(1-C'\left(\frac{\log \Delta}{\Delta}\right)^{1/2}\right)(1-t_i) d(x_{\mu_i},x_{\nu_i})\leq d(z_i,x_{\nu_i})\leq \left(1+C'\left(\frac{\log \Delta}{\Delta}\right)^{1/2}\right)(1-t_i)d(x_{\mu_i},x_{\nu_i}).$$
It is now easy to check that the sequence $x=z_0,z_1,\ldots, z_{m-1},z_m=y$ satisfies the conclusion of Proposition \ref{prop:VSG}.
 \qed

\section{Strong asymptotical geodesicity implies limit shape}\label{sec:MainProof}
This section is dedicated to the proof of ``(i) implies (ii)"
in Proposition \ref{prop:AGversusCV}. Given a subset $A$ of $\R^d$, and $t\in \R$, we denote $tA=\{ta, \; a\in A\}$, and $AB=\{a+b, \;a\in A,b\in B\}$. So in particular $A^n=\{a_1+\ldots+a_n,\; a_i\in A\}$. 
We fix a norm $\|\cdot\|_0$ on $\R^d$.  Recall that the Hausdorff distance between two compact subsets $A$ and $B$ of $\R^n$ is defined as $$d_H(A,B)=\sup\left\{r>0,\; A\subset [B]_r, B\subset[A]_r\right\},$$ where $[A]_r$ denotes the set of points of $\R^d$ at distance at most $r$ from $A$. Observe that $[A]_r=AB_{\|\cdot\|_0}(0,r).$

Note that since $d$ and $\delta$ are both bi-Lipschitz equivalent to $\|\cdot \|_0$, (ii) is equivalent to the fact that there exists a norm $\|\cdot \|$, $C>0$ and $n_2$ such that for all $n\geq n_2$, 
$$d_H\left(\frac{1}{n}B_{\delta}(0,n),B_{\|\cdot\|}(0,1)\right)\leq \frac{C}{\phi(n)}.$$
For convenience, in the sequel, we shall omit the suffix $\delta$ for the $\delta$-ball. 
Denote $\hat{B}(0,r)$ the convex hull of $B(0,r)$.

\subsection{Preliminary lemmas}
In what follows, we suppose that (i) is satisfied.
The following lemma is the only place where we actually use this assumption.
\begin{lem}\label{lem:SAG} There exists $C'$ such that for all $M\in \N$ and all $r\geq M$, 
$$d_H\left(\frac{1}{r}B(0,r/M)^M,\frac{1}{r}B(0,r)\right)\leq \frac{C'}{\phi(r/M)}.$$

\end{lem}
\begin{proof}
First, note that
$$B(0,r/M)^M\subset B(0,r)\subset B\left(0,(1+ \eps) r/M\right)^M,$$
where $$\eps=\frac{1}{N(r/M)}\leq \frac{1}{c\phi(r/M)}.$$
The left inclusion simply follows from triangular inequality, while the right inclusion results from (\ref{eq:strong}). Recall that in restriction to $\Z^d$, $\|\cdot \|_0$ and $\delta$ are bi-Lipschitz equivalent, so there exists a constant $G>0$ such that $\delta(x,y)\leq G\|x-y \|_0$ for all $x,y\in \Z^d$. 
Then, using (\ref{eq:strong'}), we have
\begin{eqnarray*}
B(0,(1+\eps) r/M)^M & \subset & \left([B(0,r/M)]_{6\eps r/M}\right)^M\\
                                 &  =          & \left(B(0,r/M)B_{\|\cdot \|_0}(0,6G\eps r/M)\right)^M\\
                                 &  =          &  \left(B(0,r/M)\right)^MB_{\|\cdot \|_0}(0,6G\eps r).            
\end{eqnarray*}                                 
Hence the lemma follows with $C'=6G/c.$
\end{proof}
We let $L$ be an integer $\geq 2$ to be determined later. Let $k\in\N$ be such that $L^k\leq r<L^{k+1}$. 
\begin{cor}\label{cor:geodesicAs}
For all $r\in \N$, there exists $C"$ (depending on $L$) such that for all $m\in \N$ with $m\geq 1$,
$$d_H\left(\frac{1}{mr} B(0,r/L)^{mL},\frac{1}{mr} B(0,mr)\right)\leq \frac{C"}{\phi(L^{k})}.$$
\end{cor}
\begin{proof}
Applying Lemma \ref{lem:SAG} yields that the left-hand term is at most $ \frac{C'}{\phi(L^{k-1})}$. Using that $\phi(L^k)\leq L\phi(L^{k-1})$, we deduce the corollary with $C"=LC'$.
\end{proof}

We now proceed to a innocent-looking lemma, that nevertheless concentrates the main feature of $\Z^d$ that is needed for the proof.

\begin{lem}\label{lem:convexhull}
Let $K$ be a compact symmetric subset of $\R^d$, and let $\hat{K}$ be its convex hull. Then, for all $n\geq d$, we have
$$\hat{K}^n= K^{n-d}\hat{K}^{d}.$$In particular, $$d_H(\widehat{K^n},\hat{K}^n)\leq d_H(K^n,\hat{K}^n)\leq  d\cdot d_H(K,\hat{K}).$$
\end{lem}
\begin{proof}
One inclusion is clear, so let us prove the other one. Let $x\in \hat{K}^n$. By convexity, $\hat{K}^n=n\hat{K}$, so that there exists $y\in \hat{K}$ such that $x=ny$. Now $y$ can be written as a convex combination $y=t_0y_0+\ldots +t_dy_d$ of $d+1$ elements of $K$. Write $nt_i=n_i+s_i$, where $s_i\in [0,1)$, and $n_i=[nt_i]$. Observe that the integer $m:=\sum_{i}s_i=n-\sum_i n_i$ satisfies $m< d+1$. Since, $\frac{1}{m}\sum_is_i e_i$ is a convex combination of the $e_i$, it belongs to $\hat{K}$. So we have $\sum_is_i e_i\in m\hat{K}=\hat{K}^m$. On the other hand, since $\sum_i n_i=n-m$, the element $\sum_i n_ie_i$ belongs to $K^{n-m}$. We deduce that
$\hat{K}^n\subset K^{n-m}\hat{K}^m\subset K^{n-d}\hat{K}^d$. 
To deduce the second inequality, we let $t=d_H(K,\hat{K})$ so that $\hat{K}\subset KB_{\|\cdot \|_0}(0,t)$. It follows (since $\R^d$ is abelian)  that 
$$\hat{K}^n\subset K^{n-d}\hat{K^d} \subset K^{n-d}K^dB_{\|\cdot \|_0}(0,dt),$$
so that $d_H(K^n,\hat{K}^n)\leq dt.$
\end{proof}
The following corollary is immediate.
\begin{cor}\label{cor:convex}
For all $L>1$, and $r,m\in \N$ with $m\geq 1$,
$$d_H\left(\frac{1}{mr}B(0,r/L)^{Lm},\frac{1}{mr}\hat{B}(0,r/L)^{Lm}\right)\leq d\cdot d_H\left(\frac{1}{r}B(0,r/L),\frac{1}{r}\hat{B}(0,r/L)\right).$$
\end{cor}

On the other hand, combining Corollary \ref{cor:geodesicAs} and Lemma \ref{lem:convexhull}, we obtain:

\begin{cor}\label{cor: geodesicAsconvexhull} For all $L>1$, and $r,m\in \N$ with $m\geq 1$,

 $$d_H\left(\frac{1}{mr} \hat{B}(0,r/L)^{mL},\frac{1}{mr} \hat{B}(0,mr)\right)\leq \frac{C"}{\phi(L^{k})}+d\cdot d_H\left(\frac{1}{r}B(0,r/L),\frac{1}{r}\hat{B}(0,r/L)\right).$$
\end{cor}
\begin{proof}
By Lemma \ref{lem:convexhull}, denoting $V$ for the convex hull of $B(0,r/L)^{mL}$, we have
$$d_H\left(\frac{1}{r}V,\frac{1}{r} \hat{B}(0,r/L)^{mL}\right)\leq d\cdot d_H\left(\frac{1}{r}B(0,r/L),\frac{1}{r}\hat{B}(0,r/L)\right).$$
On the other hand for every pair $A,A'$ of compact subsets of $\R^n$, we have that $d_H(\hat{A},\hat{A'})\leq d_H(A,A')$. Hence taking the convex hull in Corollary \ref{cor:geodesicAs} yields
$$d_H\left(\frac{1}{r} V,\frac{1}{mr} \hat{B}(0,mr)\right)\leq  \frac{C"}{\phi(L^{k})}.$$
These two inequalities now give the corollary.
\end{proof}
\subsection{Proof of (i) implies (ii) in Proposition \ref{prop:AGversusCV}.}

We let $L> 1$ to be determined later.
We will prove by induction on $k$ the following Cauchy criterion for the sequence $\frac{1}{r}B(0,r)$: there exists $C$ such that for all $k$, all $r\in \N$ such that $r\geq L^k$ and all positive integer $m$, 
\begin{equation}\label{eq:HR}
d_H\left(\frac{1}{r}B(0,r),\frac{1}{mr}\hat{B}(0,mr)\right)\leq \frac{C}{\phi(L^k)}.
\end{equation}
Note that by triangular inequality, this implies that for all $t\in \Q$ such that $t\geq 1$,   
$$d_H\left(\frac{1}{r}B(0,r),\frac{1}{tr}B(0,tr)\right)\leq  \frac{2C}{\phi(L^{k})},$$
but since $\bar{d}$ takes values in a discrete set, we deduce that for all $r',r"\geq r$,
\begin{equation}\label{eq:rr'}
d_H\left(\frac{1}{r'}B(0,r'),\frac{1}{r"}B(0,r")\right)\leq  \frac{2C}{\phi(L^{k})}.
\end{equation}

Observe that this both implies a Cauchy criterion for $\frac{1}{r}B(0,r)$ and the fact that the limit is a convex body (as it is also the limit of the sequence of convex bodies 
$\frac{1}{r}\hat{B}(0,r)$). It also gives the right rate of convergence.

Note that since $\delta$ is bilipschitz equivalent to $\|\cdot \|_0$, there exists $r_0$ such that for all $r\geq 1$,  
\begin{equation}\label{eq:r0}
\frac{1}{r}B(0,r)\subset B_{\|\cdot \|_0}(0,r_0).
\end{equation}
So in particular, $$C_0:=\sup_{1\leq r\leq L;\; r\leq r'}d_H\left(\frac{1}{r}B(0,r),\frac{1}{r'}\hat{B}(0,r')\right)<\infty.$$
We let $L$ be such that $$\frac{d\eta(L)}{L}\leq 1/4,$$  
and we let $C=\max\{C_0,C"/4\}$, where $C"$ is the constant of Corollary \ref{cor:geodesicAs}. Recall that $\eta: \R_+\to \R_{+}$ is a function that satisfies  $\lim_{\lambda\to \infty}\eta(\lambda)/\lambda=0$ and for all $\lambda>0$, $\phi(\lambda r)\leq \eta(\lambda)\phi(r)$. 
\medskip

\noindent{\bf Initial step:}
Since $C\geq C_0$, we have that (\ref{eq:HR}) holds for $k=0$. 

\medskip

\noindent{\bf Induction hypothesis:}  We let $k\geq 1$, we assume that (\ref{eq:HR}) holds for all $k'<k$, and we let $r\in \N$ such that $r\geq L^k$. 

\medskip

We have, by triangular inequality, 

\begin{eqnarray*}
d_H\left(\frac{1}{r}B(0,r),\frac{1}{mr}\hat{B}(0,mr)\right) & \leq & d_H\left(\frac{1}{r}B(0,r),\frac{1}{r}B(0,r/L)^{L}\right)\\
                                                                                        &       & + d_H\left(\frac{1}{r}B(0,r/L)^{L},\frac{1}{r}\hat{B}(0,r/L)^{L}\right)\\
                                                                                        &      &  + d_H\left(\frac{1}{r}\hat{B}(0,r/L)^L,\frac{1}{mr}\hat{B}(0,r/L)^{mL}\right)\\
                                                                                        &       & + d_H\left(\frac{1}{mr}\hat{B}(0,r/L)^{mL},\frac{1}{mr}\hat{B}(0,mr)\right)
 \end{eqnarray*}
Note that the third term is zero.
The first term is taken care of by Corollary \ref{cor:geodesicAs} (with $m=1$):  
$$d_H\left(\frac{1}{r}B(0,r),\frac{1}{r}B(0,r/L)^{L}\right) \leq \frac{C"}{\phi(L^k)}\leq  \frac{C}{4\phi(L^k)}   
$$
To deal with the second terms, we apply Corollary \ref{cor:convex} and the induction hypothesis (with $m=1$) as follows:
\begin{eqnarray*}
d_H\left(\frac{1}{r}\hat{B}(0,r/L)^{L},\frac{1}{r}B(0,r/L)^{L}\right)& \leq & d\cdot d_H\left(\frac{1}{r}\hat{B}(0,r/L),\frac{1}{r}B(0,r/L)\right)\\
                                                                                                               & \leq &   \frac{d}{L} d_H\left(\frac{1}{(r/L)}\hat{B}(0,r/L),\frac{1}{(r/L)}B(0,r/L)\right)\\
                                                                                                               & \leq &  \frac{dC}{L\phi(L^{k-1})}\\
                                                                                                               & \leq & \frac{dC \eta(L)}{L\phi(L^k)}\\
                                                                                                               & \leq & \frac{C}{4\phi(L^k)}.                                                        
 \end{eqnarray*}

To treat the fourth terms, we apply Corollary \ref{cor: geodesicAsconvexhull} and once again the induction hypothesis:
\begin{eqnarray*}
d_H\left(\frac{1}{mr}\hat{B}(0,r/L)^{mL},\frac{1}{mr}\hat{B}(0,mr)\right)& \leq & d\cdot d_H\left(\frac{1}{r}\hat{B}(0,r/L),\frac{1}{r}B(0,r/L)\right)+\frac{C}{4\phi(L^k)}\\
                                                                                                               & \leq & \frac{C}{2\phi(L^k)}                                                                                            
 \end{eqnarray*}
Combining these three inequalities proves (\ref{eq:HR}), i.e.\ that 
$$d_H\left(\frac{1}{r}B(0,r),\frac{1}{mr}\hat{B}(0,mr)\right)\leq  \frac{C}{\phi(L^k)},$$
which ends the proof that (i) implies (ii). \qed

\section{Limit shape implies strong asymptotical geodesicity}\label{sec:(ii)to(i)}

The aim of this section is to prove that (ii) implies (i) in Proposition \ref{prop:AGversusCV}. The proof is rather straightforward, so we will only prove (\ref{eq:strong}), leaving (\ref{eq:strong'}) to the reader.

Observe that since $\delta$ is an invariant distance on $\Z^d$,  (ii) is equivalent to the fact that  there exists $C>0$ and $\alpha>0$, such that for all $x,y\in \Z^d$ such that $\delta(x,y)\geq \alpha$, 
$$ \|x-y\|\left(1-\frac{C}{\phi(\delta(x,y))}\right)\leq \delta(x,y)\leq \|x-y\|\left(1+\frac{C}{\phi(\delta(x,y))}\right).$$ 
Since $\phi$ is doubling, and since for $\delta(x,y)$ large enough, $\|x-y\|/2\leq\delta(x,y)\leq 2\|x-y\|$, up to changing the constant $C$, we have
\begin{equation}\label{eq:approxdelta}
 \|x-y\|\left(1-\frac{C}{\phi( \|x-y\|)}\right)\leq \delta(x,y)\leq \|x-y\|\left(1+\frac{C}{\phi( \|x-y\|)}\right).
 \end{equation}
 
Now, fix two elements $x,y\in \Z^d$ and consider the segment $[x,y]$ in $\R^d$.  We let $m\in \N$ and consider $x=x_0,\ldots, x_m=y$ such that for all $0\leq i\leq m$, $x_i=x+i(y-x)/m$. The $z_i$ are not necessarily in $\Z^d$, so we pick for each $1\leq i\leq l-1$, some $x_i\in \Z^d$ such that $\|x_i-z_i\|\leq K$, where $K=\sup_{z\in \R^d}\inf_{v\in \Z} \|v-z\|.$ We now have a sequence $x_0=x,\ldots, x_l=y$ of points in $\Z^d$ such that 
$$\|z_i-z_{i+1}\|-2K\leq \|x_i-x_{i+1}\|\leq \|z_i-z_{i+1}\|+2K.$$
Let us assume that $\alpha=\|y-x\|/m$ is large enough, so that $$\frac{2K}{\|z_i-z_{i+1}\|}=\frac{2K}{\alpha}\leq \frac{C}{\phi(\alpha)}.$$ This is possible thanks to the fact that $\lim_{r\to \infty}\phi(r)/r=0$.We deduce that 
\begin{equation}\label{eq:xizi}
\left(1-\frac{C}{\phi(\alpha)}\right)\|z_i-z_{i+1}\|\leq \|x_i-x_{i+1}\|\leq\left(1+\frac{C}{\phi(\alpha)}\right)\|z_i-z_{i+1}\|.
\end{equation}
Since $\phi$ is increasing, we deduce from (\ref{eq:approxdelta}) that
$$ \left(1-\frac{C}{\phi(\alpha)}\right)\|x-y\|\leq  \delta(x,y)  \leq  \left(1+\frac{C}{\phi(\alpha)}\right)\|x-y\|,$$
and
$$ \left(1-\frac{C}{\phi(\alpha)}\right)\|x_i-x_{i+1}\|\leq  \delta(x_i,x_{i+1})  \leq  \left(1+\frac{C}{\phi(\alpha)}\right)\|x_i-x_{i+1}\|.$$
Combining these two inequalities with (\ref{eq:xizi}), and assuming  $\alpha$ large enough so that $C/\phi(\alpha)\leq 1/2$, we deduce 
$$ \left(1-\frac{C}{\phi(\alpha)}\right)^2\left(1+\frac{C}{\phi(\alpha)}\right)^{-1} \delta(x,y)/m \leq \delta(x_i,x_{i+1})  \leq  \left(1+\frac{C}{\phi(\alpha)}\right)^2\left(1-\frac{C}{\phi(\alpha)}\right)^{-1} \delta(x,y)/m.$$
Hence,
$$ \left(1-4\frac{C}{\phi(\alpha)}\right) \delta(x,y)/m \leq \delta(x_i,x_{i+1})  \leq  \left(1+5\frac{C}{\phi(\alpha)}\right)\delta(x,y)/m,$$
 so (\ref{eq:strong}) follows. \qed

\appendix

\section{Lower bound on $\bar{d}$}
In this section, we prove that  a mild assumption on $\nu$ implies $(A_2)$.
\begin{lem}\label{lem:average}
Let $X=(V,E)$ be a graph of degree $\leq q$. Assume that $\nu$ is supported on $[a,\infty)$ and that $\nu(\{a\})<1/q$. Then there exists $a'>a$ and $r_0$ such that $\bar{d}(x,y)\geq a'd(x,y)$ for all $x,y\in V$ such that $d(x,y)\geq r_0$. 
\end{lem}

\begin{proof}
For simplicity, let us assume that $a=0$. The assumption implies that there exists $\delta>0$ such that $\lambda:=\nu([0,0+\delta])<1/q$. Let $\eps>0$ to be determined later.
Let $\gamma$ be an  $\omega$-geodesic between $x$ and $y$ with $\omega$-length $\leq \eps d(x,y)$, and with length $n$ (note that $n\geq d(x,y)$). Assume that such a path admits $N$ edges of $\omega$-length $\geq \delta d(x,y)$.  It follows that
$$\delta N \leq \eps n,$$
so we deduce that $N\leq \eps n/\delta$.
This imposes that  at least $(1-\eps/\delta)n$ edges of $\gamma$ have $\omega$-length $\leq \delta$.
Recall that by Stirling's formula, given some $0<\alpha<1$, the number of ways to choose $\alpha n$ edges in a path of length n is  $$\sim \frac{n^n}{(\alpha n)^{\alpha n}((1-\alpha) n)^{(1-\alpha) n}}= (1/\alpha)^{\alpha n}(1/(1-\alpha)^{(1-\alpha)n}.$$
Thus the probability that $\gamma$  has $\omega$-length at most $ \eps n$ is  less than a universal constant times
$$\frac{\lambda^{(1-\eps/\delta)n}}{(\eps/\delta)^{(\eps/\delta)n}(1-\eps/\delta)^{(1-\eps/\delta)n}}=\left(\frac{\lambda^{1-\eps/\delta}}{(\eps/\delta)^{\eps/\delta}(1-\eps/\delta)^{1-\eps/\delta}}\right)^{n}.$$

Note that $$\lim_{\eps\to 0}\frac{\lambda^{1-\eps/\delta}}{(\eps/\delta)^{\eps/\delta}(1-\eps/\delta)^{1-\eps/\delta}}=\lambda.$$
 On the other hand, the number of paths of length $\leq n$ is at most $q^n$. We deduce that for this choice of $\eps$, the probability that $d_{\omega}(x,y)\leq (1+\eps)a d(x,y)$ is at most a constant times $(\lambda' q)^{n}\leq (\lambda' q)^{d(x,y)}$, which converges to $0$ as  $d(x,y)\to \infty$. This proves the lemma.
\end{proof}

\begin{cor}\label{cor:degree}
 Let $X=(V,E)$ be a graph of degree $\leq q$. We assume that $(A_1)$ is satisfied, and that $\nu(\{0\})<1/q$. Then there exists $a">0$ such that $\bar{d}(x,y)\geq a"d(x,y)$ for all $x,y\in X$.
\end{cor}
\begin{proof}
By the previous lemma, applied with $a=0$, there exists $r_0$, and some $a'>0$ such that $\bar{d}(x,y)\geq a'd(x,y)$ as soon as $d(x,y)\geq r_0$.
On the other hand, the assumption implies that there exists $\delta>0$ such that $c:=\nu([0,\delta])<1$. Since the degree is at most $k$, 
the probability that the $\omega$-length of all vertices issued from a given vertex is at least $\delta$, is at least $(1-c)^k$. Hence the distance between two distinct points is $\geq (1-c)^k\delta$. The corollary follows by taking $a"=\min\{a',(1-c)^k\delta/r_0\}$. 
\end{proof}


\bigskip
\footnotesize

\end{document}